\documentclass[twoside,leqno,10pt, A4]{amsart}
\usepackage{amsfonts}
\usepackage{amsmath}
\usepackage{amscd}
\usepackage{amssymb}
\usepackage{amsthm}
\usepackage{amsrefs}
\usepackage{latexsym}
\usepackage{mathrsfs}
\usepackage{bbm}
\usepackage{amscd}
\usepackage{amssymb}
\usepackage{amsthm}
\usepackage{amsrefs}
\usepackage{latexsym}
\usepackage{mathrsfs}
\usepackage{bbm}
\usepackage{enumerate}
\usepackage{graphicx}
\usepackage{color}
\begin{document}

\newtheorem{theorem}[subsection]{Theorem}
\newtheorem{proposition}[subsection]{Proposition}
\newtheorem{lemma}[subsection]{Lemma}
\newtheorem{corollary}[subsection]{Corollary}
\newtheorem{conjecture}[subsection]{Conjecture}
\newtheorem{prop}[subsection]{Proposition}
\newtheorem{defin}[subsection]{Definition}

\numberwithin{equation}{section}
\newcommand{\mr}{\ensuremath{\mathbb R}}
\newcommand{\mc}{\ensuremath{\mathbb C}}
\newcommand{\dif}{\mathrm{d}}
\newcommand{\intz}{\mathbb{Z}}
\newcommand{\ratq}{\mathbb{Q}}
\newcommand{\natn}{\mathbb{N}}
\newcommand{\comc}{\mathbb{C}}
\newcommand{\rear}{\mathbb{R}}
\newcommand{\prip}{\mathbb{P}}
\newcommand{\uph}{\mathbb{H}}
\newcommand{\fief}{\mathbb{F}}
\newcommand{\majorarc}{\mathfrak{M}}
\newcommand{\minorarc}{\mathfrak{m}}
\newcommand{\sings}{\mathfrak{S}}
\newcommand{\fA}{\ensuremath{\mathfrak A}}
\newcommand{\mn}{\ensuremath{\mathbb N}}
\newcommand{\mq}{\ensuremath{\mathbb Q}}
\newcommand{\half}{\tfrac{1}{2}}
\newcommand{\f}{f\times \chi}
\newcommand{\summ}{\mathop{{\sum}^{\star}}}
\newcommand{\chiq}{\chi \bmod q}
\newcommand{\chidb}{\chi \bmod db}
\newcommand{\chid}{\chi \bmod d}
\newcommand{\sym}{\text{sym}^2}
\newcommand{\hhalf}{\tfrac{1}{2}}
\newcommand{\sumstar}{\sideset{}{^*}\sum}
\newcommand{\sumprime}{\sideset{}{'}\sum}
\newcommand{\sumprimeprime}{\sideset{}{''}\sum}
\newcommand{\sumflat}{\sideset{}{^\flat}\sum}
\newcommand{\shortmod}{\ensuremath{\negthickspace \negthickspace \negthickspace \pmod}}
\newcommand{\V}{V\left(\frac{nm}{q^2}\right)}
\newcommand{\sumi}{\mathop{{\sum}^{\dagger}}}
\newcommand{\mz}{\ensuremath{\mathbb Z}}
\newcommand{\leg}[2]{\left(\frac{#1}{#2}\right)}
\newcommand{\muK}{\mu_{\omega}}
\newcommand{\thalf}{\tfrac12}
\newcommand{\lp}{\left(}
\newcommand{\rp}{\right)}
\newcommand{\Lam}{\Lambda_{[i]}}
\newcommand{\lam}{\lambda}
\newcommand{\af}{\mathfrak{a}}
\newcommand{\sw}{S_{[i]}(X,Y;\Phi,\Psi)}
\newcommand{\lz}{\left(}
\newcommand{\pz}{\right)}
\newcommand{\bfrac}[2]{\lz\frac{#1}{#2}\pz}
\newcommand{\odd}{\mathrm{\ primary}}
\newcommand{\even}{\text{ even}}
\newcommand{\res}{\mathrm{Res}}
\newcommand{\sumn}{\sumstar_{(c,1+i)=1}  w\left( \frac {N(c)}X \right)}
\newcommand{\lab}{\left|}
\newcommand{\rab}{\right|}
\newcommand{\Go}{\Gamma_{o}}
\newcommand{\Ge}{\Gamma_{e}}
\newcommand{\M}{\widehat}
\def\su#1{\sum_{\substack{#1}}}
\newcommand{\sgn}{\text{sgn}}

\theoremstyle{plain}
\newtheorem{conj}{Conjecture}
\newtheorem{remark}[subsection]{Remark}

\newcommand{\pfrac}[2]{\left(\frac{#1}{#2}\right)}
\newcommand{\pmfrac}[2]{\left(\mfrac{#1}{#2}\right)}
\newcommand{\ptfrac}[2]{\left(\tfrac{#1}{#2}\right)}
\newcommand{\pMatrix}[4]{\left(\begin{matrix}#1 & #2 \\ #3 & #4\end{matrix}\right)}
\newcommand{\ppMatrix}[4]{\left(\!\pMatrix{#1}{#2}{#3}{#4}\!\right)}
\renewcommand{\pmatrix}[4]{\left(\begin{smallmatrix}#1 & #2 \\ #3 & #4\end{smallmatrix}\right)}
\def\en{{\mathbf{\,e}}_n}

\newcommand{\ppmod}[1]{\hspace{-0.15cm}\pmod{#1}}
\newcommand{\ccom}[1]{{\color{red}{Chantal: #1}} }
\newcommand{\acom}[1]{{\color{blue}{Alia: #1}} }
\newcommand{\alexcom}[1]{{\color{green}{Alex: #1}} }
\newcommand{\hcom}[1]{{\color{brown}{Hua: #1}} }

\makeatletter
\def\widebreve{\mathpalette\wide@breve}
\def\wide@breve#1#2{\sbox\z@{$#1#2$}%
     \mathop{\vbox{\m@th\ialign{##\crcr
\kern0.08em\brevefill#1{0.8\wd\z@}\crcr\noalign{\nointerlineskip}%
                    $\hss#1#2\hss$\crcr}}}\limits}
\def\brevefill#1#2{$\m@th\sbox\tw@{$#1($}%
  \hss\resizebox{#2}{\wd\tw@}{\rotatebox[origin=c]{90}{\upshape(}}\hss$}
\makeatletter

\title[Mean squares of Quadratic twists of the Fourier coefficients of Modular forms]{Mean squares of Quadratic twists of the Fourier coefficients of Modular forms}

%%\date{\today}
\author[P. Gao]{Peng Gao}
\address{School of Mathematical Sciences, Beihang University, Beijing 100191, China}
\email{penggao@buaa.edu.cn}

\author[Y. Zhao]{Yuetong Zhao}
\address{School of Mathematical Sciences, Beihang University, Beijing 100191, China}
\email{yuetong.zhao.math@gmail.com}

\begin{abstract}
 In this paper, we evaluate asymptotically a smoothed version of the sum
    \[ \displaystyle \sumstar_{d \leq X} \left( \sum_{n \leq Y} \lambda_f(n)\Big(\frac{8d}{n}\Big)\right)^2, \]
    where $\leg {8d}{\cdot}$ is the Kronecker symbol, $\sumstar$ denotes a sum over positive odd square-free integers and $\lambda_f(n)$ are Fourier coefficients of a given modular form $f$.
\end{abstract}

\maketitle



\section{Introduction}\label{sec 1}

   Let $f$ be a fixed holomorphic Hecke eigenform of weight $\kappa\equiv 0 \pmod{4}$ for the full modular group $SL_2(\mz)$. We write the Fourier expansion of  $f$ at infinity as
\begin{align}
\label{fFourier}
f(z)=\sum_{n=1}^{\infty}\lambda_f(n)n^{\frac{\kappa-1}{2}}e(nz),
\end{align}
where $e(z)=e^{2\pi iz}$. It follows from Deligne’s proof \cite{D} of the Weil conjecture that we have
\begin{align}
\label{lambdabound}
 |\lambda_f(n)|\leq d(n)
\end{align}
 where $d(n)$ is the  divisor function of $n$.

Considerable research has been devoted to understanding the behavior of Fourier coefficients of cusp forms. For example, it is shown by J. L. Hafner and Ivi\'c in \cite{HI89} that
\begin{align*}
\sum_{n\leq x}\lambda_f(n)\ll x^{1/3}.
\end{align*}

  In \cite{G&Zhao2024-12}, P. Gao and L. Zhao studied bounds for moments of sums involving with $\lambda_f (n)$ twisted by $\chi_{8d}(n)$ for $d$ running over odd, positive, square-free integers. Here we write $\chi_{d}$ for the Kronecker symbol $\big(\frac{d}{\cdot}\big)$ and we note that $\chi_{8d}$ is a primitive Dirichlet character when $d$ is odd and square-free. More precisely, for any real number $m>0$, let
\begin{align*}
%%\label{mainestimation}
 S_{m}(X,Y;f) =:  \sumstar_{\substack{d \leq X }}\Big | \sum_{n \leq Y}\chi_{8d}(n)\lambda_f(n)\Big |^{m},
\end{align*}
  where $\sumstar$ denotes a sum over positive odd square-free integers. Then it is shown in \cite[Theorem 1.6]{G&Zhao2024-12} that under the generalized Riemann hypothesis (GRH) that for $m \geq 4$, one has
\begin{align}
\label{Sest}
 S_{m}(X,Y;f) \ll XY^{m/2}(\log X)^{\frac {m(m-3)}{2}+2}.
\end{align}
  The above result is achieved using the method developed by B. Szab\'o in \cite{Szab} to apply sharp upper bounds for shifted moments of families of $L$-functions to obtain bounds for moments of the corresponding character sums. Here we also note that the establishment of sharp upper bounds for shifted moments of families of $L$-functions under GRH relies crucially on a method of K. Soundararajan \cite{Sound2009} and its refinement by A. J. Harper \cite{Harper}.

  Similar to \eqref{Sest}, it is shown in \cite[Theorem 1.3]{G&Zhao2024-3} that under GRH, for any real $m>\sqrt{5}+1$,
\begin{align}
\label{momentDCS}
 \sumstar_{\substack{d \leq X }}\Big | \sum_{n \leq Y}\chi_{8d}(n)\Big |^{m} \ll XY^{m/2}(\log X)^{\frac {m(m-1)}{2}+1}.
\end{align}
  It is further proved by M.  Munsch in \cite{Munsch25} that for a smoothed version of the left-hand side expression above, one may replace under GRH the exponent $\frac {m(m-1)}{2}+1$ by $\frac {m(m-1)}{2}$ of the right-hand side expression above for integers $m \geq 2$ and the resulting estimation is optimal.

  The result of Munsch is in accordance to an earlier result of M. V. Armon \cite[Theorem 2]{Armon}, who obtained unconditionallly the optimal bound for the left-hand side expression of \eqref{momentDCS} for the case $m=1$. Note more over that in \cite{Gao19}, the first named author evaluated the expression
\begin{align*}
%%\label{SXY}
   \sumstar_{\substack {d}} \left| \sum_{n} \chi_{8d}(n)\Phi \left(\frac {n}{Y} \right ) \right |^2 \Psi\left(\frac {d}{X} \right ),
\end{align*}
   where $\Phi$ and $\Psi$ are smooth, compactly supported functions on ${\mr}_+=(0,\infty)$, having
 all partial derivatives extending continuously to the boundary, and satisfying the partial
 derivative bounds
\begin{align}
\label{dercond}
\begin{split}
 y^j\Phi^{(j)} \left(\frac {y}{Y} \right ) \ll (1+\frac {y}{Y})^{-A}, \quad  x^j\Psi^{(j)} \left(\frac {x}{X} \right ) \ll (1+\frac {x}{X})^{-B},
\end{split}
\end{align}
  for any positive integers $A, B$.

 An asymptotic formula for the above expression is given in \cite[Theorem 1.1]{Gao19}, which is valid when $X^{1/2} \ll Y \ll X^{1-\varepsilon}$ for any $\varepsilon>0$.

   Motivated by the results in \cite{Gao19} and in \cite{G&Zhao2024-12}. It is the aim of this paper to evaluate asymptotically a smoothed version of $S_{m}(X,Y;f)$. For this, we fix two non-negative, smooth and compacted supported functions $\Phi, \Psi$ as mentioned above and we define
\begin{align} \label{SXYPW}
    S_f(X,Y; \Phi, \Psi)=\sumstar_{d} \Bigg(\sum_{n} \lambda_f(n)\chi_{8d}(n)\Phi \Big(\frac nY \Big)\Bigg)^2 \Psi \Big(\frac d{X} \Big).
\end{align}

   Our main result evaluates $S_f(X,Y; \Phi, \Psi)$ asymptotically as follows.
\begin{theorem}\label{meansquare}
   With the notation as above. We have for large $X$ and $Y$ and any $\varepsilon>0$,
\begin{align}
\label{S}
   S_f(X,Y; \Phi, \Psi)=C_0(\Phi, \Psi)XY+O(X^{1/2+\varepsilon}Y^{3/2+\varepsilon}+XY^{1/2+\varepsilon}),
\end{align}
  where $C_0(\Phi, \Psi)$ is given in \eqref{C0def}, which is a constant depending on $\Phi$ and $\Psi$ only.
\end{theorem}

   We note that as often is the case, the most interesting case for evaluating $S_f(X,Y; \Phi, \Psi)$ is when $X$ and $Y$ are of compatible size $X \asymp Y$.   Here one checks that \eqref{S} yields a valid asymptotic formula when $Y \ll X^{1-\varepsilon}$ for any $\varepsilon>0$.  Note also in view of \eqref{Sest} and the result of Theorem \ref{meansquare} also suggests that one has for all real $m \geq 2$,
\begin{align*}
%%\label{Sestsharp}
 S_{m}(X,Y;f) \ll XY^{m/2}(\log X)^{\frac {m(m-3)}{2}+1}.
\end{align*}
  Our proof of Theorem \ref{meansquare} makes use of the Poisson summation formula established in \cite{sound1} and follows closely the treatments developed by  K. Soundararajan and M. P. Young \cite{S&Y} in their work on the second moment of quadratic twists of modular $L$-functions.

\section{Preliminaries}

In this section, we gather a few auxiliary results needed in the proof of Theorem \ref{meansquare}.
\subsection{Gauss sums}
For all odd integers $m$ and all integers $k$, we define the Gauss-type sums $G_k(m)$, as did in Section 2.2 in \cite{sound1}, such that
\begin{align}\label{G}
    G_k(m)=
    \left( \frac {1-i}{2}+\left( \frac {-1}{m} \right)\frac {1+i}{2}\right)\sum_{a \mod m}\left( \frac {a}{m} \right) e \left( \frac {ak}{m} \right).
\end{align}

   We reserve the letter $p$ for a prime number throughout the paper. Let also $\varphi(m)$ be the Euler totient function of $m$. Our next result  evaluates $G_k(m)$, which is taken from \cite[Lemma 2.3]{sound1}.
\begin{lemma} \label{Gausssum}
       If $(m_1,m_2)=1$ then $G_k(m_1m_2)=G_k(m_1)G_k(m_2)$. Suppose that $p^a$ is
       the largest power of $p$ dividing $k$ (put $a=\infty$ if $k=0$).
       Then for $b \geq 1$ we have
    \begin{equation*}% \label{011}
        G_k(p^b)= \left\{\begin{array}{cl}
        0  & \mbox{if $b\leq a$ is odd}, \\
        \varphi(p^b) & \mbox{if $b\leq a$ is even},  \\
        -p^a  & \mbox{if $b=a+1$ is even}, \\
        (\frac {k/p^a}{p})p^a\sqrt{p}  & \mbox{if $b=a+1$ is odd}, \\
        0  & \mbox{if $b \geq a+2$}.
        \end{array}\right.
    \end{equation*}
    \end{lemma}

    \subsection{Poisson Summation}
    For any smooth function $F$, we define
\begin{equation} \label{tildedef}
   \widetilde{F}(\xi)=\int\limits^{\infty}_{-\infty}\left(\cos(2\pi \xi
   x)+\sin(2\pi \xi x) \right)F(x) dx.
\end{equation}

    We note the following Poisson summation formula from \cite[Lemma 2.6]{sound1}.
\begin{lemma}\label{Posum}
   Let $F$ be a smooth function compactly supported on ${\mr}_+$. We have, for any odd integer $n$,
\begin{equation*}
\label{013}
  \sum_{(d,2)=1}\left( \frac {d}{n} \right)
    F\left( \frac {d}{X} \right)=\frac {X}{2n}\left( \frac {2}{n} \right)
    \sum_k(-1)^kG_k(n)\widetilde{F}\left( \frac {kX}{2n} \right),
\end{equation*}
where $\widetilde{F}$ is defined in \eqref{tildedef} and $G_k(n)$ is defined in \eqref{G}.
\end{lemma}

\subsection{Symmetric square $L$-functions}
\label{sec:cusp form}

    Recall that $f$ is a fixed holomorphic Hecke eigenform of weight $\kappa\equiv 0 \pmod{4}$ for the full modular group $SL_2(\mz)$ and that $\lambda_f(n)$ is given in \eqref{fFourier}. The symmetric square $L$-function $L(s, \operatorname{sym}^2 f)$ of $f$ is then defined for $\Re(s)>1$ by
 (see \cite[p. 137]{iwakow} and \cite[(25.73)]{iwakow})
\begin{align*}
%%\label{Lsymexp}
\begin{split}
 L(s, \operatorname{sym}^2 f)=& \zeta(2s) \sum_{n \geq 1}\frac {\lambda_f(n^2)}{n^s}=\prod_{p}\Big( 1-\frac {\lambda_f(p^2)}{p^s}+\frac {\lambda_f(p^2)}{p^{2s}}-\frac {1}{p^{3s}} \Big)^{-1},
\end{split}
\end{align*}
  where $\zeta(s)$ is the Riemann zeta function.

It follows from a result of G. Shimura \cite{Shimura} that $L(s, \operatorname{sym}^2 f)$ is holomorphic at $s=1$. Moreover, the corresponding completed symmetric square $L$-function
\begin{align}
\label{SymsquareLfeqn}
 \Lambda(s, \operatorname{sym}^2 f)=& \pi^{-3s/2}\Gamma \Big(\frac {s+1}{2}\Big)\Gamma \Big(\frac {s+\kappa-1}{2}\Big) \Gamma \Big(\frac {s+\kappa}{2}\Big) L(s, \operatorname{sym}^2 f)
\end{align}
  is entire and satisfies the functional equation $\Lambda(s, \operatorname{sym}^2 f)=\Lambda(1-s, \operatorname{sym}^2 f)$.

\subsection{Analytical behaviors of a Dirichlet Series}
   Let $q$ be an integer and $\alpha, \beta, \gamma \in \mc$. We define for any square-free $k_1$,
\begin{equation}\label{eq:Z}
 Z(\alpha,\beta,\gamma;q,k_1) = \sum_{k_2=1}^{\infty} \sum_{(n_1,2q)=1} \sum_{(n_2,2q)=1} \frac{\lambda_f(n_1)\lambda_f(n_2)}{n_1^{\alpha} n_2^{\beta} k_2^{2\gamma}} \frac{G_{k_1 k_2^2}(n_1 n_2)}{n_1 n_2},
\end{equation}
 where $G_k(m)$ be defined as in \eqref{G}. We write $L_c(s, \chi)$ for the Euler product of $L(s, \chi)$ with the factors from $p | c$ removed.  Our next lemma, taken from \cite[ Lemma 3.3]{S&Y}, describes the analytical behavior of $Z$.
\begin{lemma}\label{lemma:Z}
  With the notation as above. The function $Z(\alpha,\beta,\gamma;q,k_1)$ defined in \eqref{eq:Z} may be written as
 \begin{align*}
 \frac{L_q(1/2+\alpha, f\otimes\chi_{k_1})L_q(1/2+\beta,f\otimes\chi_{k_1})}{\zeta_q(1+\alpha+\beta)L_q(1+2\alpha,\operatorname{sym}^2f)L_q(1+\alpha+\beta,\operatorname{sym}^2f)L_q(1+2\beta,\operatorname{sym}^2f)}Z_{2}(\alpha,\beta,\gamma;q, k_1),
\end{align*}
where $Z_{2}(\alpha,\beta,\gamma;q,k_1)$ is a function uniformly bounded in the region
$\Re(\gamma) \geq 1/2 +\varepsilon$ and $\Re(\alpha)$, $\Re(\beta) \geq \varepsilon$ for any $\varepsilon >0$.
\end{lemma}

\subsection{Bounds for the second moment of quadratic twists of modular $L$-functions}
  We quote from \cite[Corollary 2.5]{S&Y} the following result concerning upper bounds for the second moment of quadratic twists of modular $L$-functions.
\begin{lemma}\label{MLL2}
 For $\sigma\geq 1/2$ and any $\varepsilon > 0$, we have
\begin{equation*}
 \sumflat_{|d|\leq X}\big|L(\sigma+it, f\otimes\chi_d)\big|^{2} \ll_{\varepsilon} \big(X(1+|t|)\big)^{1+\varepsilon},
\end{equation*}
  where $\sumflat$ denotes a sum over fundamental discriminants.
\end{lemma}

\section{Proof of Theorem 1.1}
Expanding the square in (\ref{SXYPW})  allows us to recast $S(X,Y;\Phi,\Psi)$ as
\begin{align*}
    S(h)
    :=\sumstar_{\substack {(d,2)=1}}\sum_{(n_1,2)=1}\sum_{(n_2,2)=1} \lambda_f(n_1)\lambda_f(n_2)\chi_{8d}(n_1n_2)h(d,n_1,n_2),
\end{align*}
where
\begin{align*}
%%\label{funch}
    h(x,y,z)=\Psi\Big(\frac{x}{X}\Big)\Phi \Big(\frac {y}Y \Big)\Phi \Big(\frac {z}Y \Big).
\end{align*}

  We denote $\mu$ for M\"obius function and apply the M\"obius inversion to remove the square-free condition on $d$ to obtain that, for a parameter $Z$ to be specified later,
\begin{align}\label{S(h)}
    S(h)
    &=\Big(\sum_{\substack{a\leq Z\\(a,2)=1}}+\sum_{\substack{a> Z\\(a,2)=1}}\Big)\mu(a)\sum_{\substack {(d,2)=1}}\sum_{(n_1,a)=1}\sum_{(n_2,a)=1} \lambda_f(n_1)\lambda_f(n_2)\chi_{8d}(n_1n_2)h(da^2,n_1,n_2),\nonumber\\
    &=: S_1(h)+S_2(h).
\end{align}

   Our next result estimates $S_2(h)$.
\begin{lemma}\label{S2h}
    We have
    \begin{align*}
    S_2(h)\ll (XY)^{1 + \varepsilon} Z^{-1+ \varepsilon}.
    \end{align*}
\end{lemma}
\begin{proof}
    We  write $d=b^2l$ with $l$ square-free, and group terms according to $c=ab$. Thus
    \begin{equation*}
         S_2(h) = \sum_{(c,2)=1} \sum_{\substack{a > Z \\ a|c}} \mu(a)
         \sumstar_{(l,2)=1}\sum_{(n_1,c)=1} \sum_{(n_2,c)=1}
         \chi_{8l}(n_1n_2) \lambda_f(n_1)\lambda_f(n_2) h(c^2l,n_1,n_2).
    \end{equation*}
 By \eqref{dercond} and  applying Mellin transforms in the variables $n_1$ and $n_2$, we see that
        \begin{align}\label{S2(h)}
        S_2(h)&=\frac{1}{(2\pi i)^2}  \sum_{(c,2)=1} \sum_{\substack{a > Z \\ a|c}} \mu(a)
        \sumstar_{(l,2)=1} \int_{(1+\varepsilon)} \int_{(1+\varepsilon)} \sum_{\substack{n_1, n_2 \\ (n_1n_2, c)=1}}\frac{\chi_{8l}(n_1)\chi_{8l}(n_2)\lambda_f(n_1)\lambda_f(n_2)}{n_1^{u}n_2^{v}}{\check h}(c^2l;u,v)du\,dv,
        \end{align}
        where
 \begin{equation}\label{checkh}
        {\check h}(x;u,v) = \int_0^{\infty} \int_0^{\infty} h(x,y,z) y^u z^v \frac{d y}{y} \frac{d z}{z}.
    \end{equation}
     Integrating by parts several times and applying \eqref{dercond}, we see that for $\Re(u)$, $\Re(v) >0$ and any positive integers $A_j$, $1 \leq j \leq 3$,
 \begin{equation}\label{uppercheckh}
    {\check h}(x;u,v) \ll \left( 1 + \frac{x}{X} \right)^{-A_1} \frac{Y^{\Re(u)+\Re(v)}}{|uv|(1+|u|)^{A_2} (1 + |v|)^{A_3}}.
\end{equation}
Note that the sum over $n_1$ and $n_2$ in \eqref{S2(h)} equals $L_c(u,f\otimes\chi_{8l}) L_c(v, f\otimes\chi_{8l})$ and we can thus
move the lines of integration in \eqref{S2(h)}
to $\Re(u)=\Re(v)=1/2+\varepsilon$ without encountering any poles.  Moreover,
\begin{align} \label{Lbound}
        |L_c(u,f\otimes\chi_{8l}) L_c(v, f\otimes\chi_{8l})|
        \leq d(c)^2 ( |L_c(u,f\otimes\chi_{8l})|^2 + |L_c(v, f\otimes\chi_{8l})|^2).
\end{align}
We now apply \eqref{uppercheckh} and \eqref{Lbound} with $A_2=A_3=5$ and $A_1$ sufficiently large to deduce from \eqref{S2(h)} that
\begin{align*}
\begin{split}
S_2(h)
 \ll & \sum_{(c,2)=1} \sum_{\substack{a > Z \\ a|c}} d(c)^2 Y^{1+2\varepsilon} \int\limits_{-\infty}^{\infty} (1+|t|)^{-10} \sumstar_{(l,2)=1} \left(1 + \frac{c^2 \ell}{X}\right)^{-A_1}
 \Big|L( 1/2+\varepsilon +it,  f\otimes\chi_{8\ell})\Big|^{2}  d t \\
\ll & (XY)^{1+\varepsilon} \sum_{(c,2)=1} \sum_{\substack{a > Z \\ a|c}}d(c)^2  /c^2\ll (XY)^{1 + \varepsilon} Z^{-1+ \varepsilon},
\end{split}
\end{align*}
 where the last estimation above follows from Lemma \ref{MLL2}. This completes the proof of the lemma.
\end{proof}

  It now remains to evaluate $S_1(h)$. Interchanging the order of summation and  in the expression of $S_1(h)$ and applying the Poisson summation formula, Lemma \ref{Posum}, to the sum over $d$, we obtain that
\begin{align}
\label{S1(h)}
\begin{split}
    S_1(h)
    =&\sum_{\substack{a\leq Z\\(a,2)=1}}\mu(a)\sum_{(n_1,a)=1}\sum_{(n_2,a)=1}\Bigg(\sum_{\substack {(d,2)=1}} \chi_{8d}(n_1n_2)\Psi\Big(\frac{d}{X}\Big)\Bigg)\lambda_f(n_1)\lambda_f(n_2)\Phi\Big(\frac{n_1}{Y}\Big)\Phi\Big(\frac{n_2}{Y}\Big)\\
    =&\sum_{\substack{a\leq Z\\(a,2)=1}}\frac{\mu(a)}{a^2}\sum_{\substack{(n_1,2a)=1\\(n_2,2a)=1}}\Bigg(\frac{X}{2n_1n_2}\sum_{k\in\mz}(-1)^{k}G_k(n_1n_2)\widetilde{\Psi}
\Big(\frac{kX}{2a^2n_1n_2}\Big)\Bigg)
    \lambda_f(n_1)\lambda_f(n_2)\Phi \Big(\frac {n_1}Y \Big)\Phi \Big(\frac {n_2}Y \Big) \\
=& \frac{X}{2}\sum_{\substack{a\leq Z\\(a,2)=1}}\frac{\mu(a)}{a^2}\sum_{k\in\mz}(-1)^{k}\sum_{\substack{(n_1,2a)=1\\(n_2,2a)=1}}\frac{\lambda_f(n_1)\lambda_f(n_2)}{n_1n_2}G_k(n_1n_2)\int_{0}^{\infty}(C+S)\Big(\frac{\pi kxX}{a^2n_1n_2}\Big)h(xX,n_1,n_2)\,dx,
\end{split}
\end{align}
 where the last expression above follows from \eqref{tildedef} upon writing $C=\cos$ and $S=\sin$.

 The main contribution to $S_1(h)$ arises from the term $k=0$ term, which we denote by $S_{k=0}$. The remaining contribution from
$k\neq 0$ is written as $S_{k\neq 0}$.

\subsection{The main term}

    We denote the symbol $\square$ for a perfect square and we apply Lemma \ref{Gausssum} to see that
\begin{equation*}% \label{011}
    G_0(n)= \left\{\begin{array}{cl}
    \varphi(n)  & \mbox{if $n=\square$ }, \\
    0 & \mbox{if $n\neq\square$ }.
    \end{array}\right.
\end{equation*}
  We then set $k=0$ in \eqref{S1(h)} to see that
\begin{equation*}
    S_{k=0}=\frac{X}{2}\sum_{\substack{(n_1n_2,2)=1\\n_1n_2=\square}}\Bigg(\sum_{\substack{a\leq Z\\(a,2n_1n_2)=1}}\frac{\mu(a)}{a^2}\Bigg)\frac{\varphi(n_1n_2)\lambda_f(n_1)\lambda_f(n_2)}{n_1n_2}H_0(n_1,n_2),
\end{equation*}
where
\begin{equation*}
    H_0(y,z)=\int_{0}^{\infty} h(xX,y,z)\,dx.
\end{equation*}
Note that
\begin{align*}
\sum_{\substack{a \leq Z \\ (a,2n_1n_2)=1}} \frac{\mu(a)}{a^2}
=& \frac{1}{\zeta(2)}
\prod_{p|2n_1n_2} \left( 1-\frac{1}{p^2}\right)^{-1} +O(Z^{-1}) = \frac{8}{\pi^2}\prod_{p|n_1n_2} \left(1-\frac{1}{p^2}\right)^{-1}+ O(Z^{-1}) .
\end{align*}
  Note moreover that
\begin{align*}
    \frac{\varphi(n_1n_2)}{n_1n_2}=\prod_{p|n_1n_2}\frac{p-1}{p},
\end{align*}
  where we denote any empty product to be $1$. It then follows that
\begin{align}\label{Sk=01}
    S_{k=0}&=\frac{4X}{\pi^2}\sum_{\substack{(n_1n_2,2)=1\\n_1n_2=\square}}\lambda_f(n_1)\lambda_f(n_2)\prod_{p|n_1n_2} \left(\frac{p}{p+1}\right)H_0(n_1,n_2)+O\Bigg(\frac{X}{Z}\sum_{\substack{(n_1n_2,2)=1\\n_1n_2=\square}}\Big|\lambda_f(n_1)\lambda_f(n_2)H_0(n_1,n_2)\Big|\Bigg).
\end{align}
  Observe from the definition of $h$ and \eqref{dercond} that $H_0\ll 1$ and $H_0=0$ unless both $n_1$ and $n_2$ are $\ll Y$. Also note that it follows from \cite[Theorem 2.11]{MVa1} that for any $\varepsilon>0$, we have
\begin{align*}
%%\label{dbound}
 d(n) \ll n^{\varepsilon}.
\end{align*}

  We deduce from the above and \eqref{lambdabound} that the $O$-term in \eqref{Sk=01} is
\begin{align}\label{Oterm}
\ll XZ^{-1}\sum_{\substack{(n_1n_2,2)=1\\n_1,n_2\leq Y\\n_1n_2=\square}}|\lambda_f(n_1)\lambda_f(n_2)|\ll XZ^{-1}\sum_{\substack{(n_1n_2,2)=1\\n_1,n_2\leq Y\\n_1n_2=\square}}d(n_1)d(n_2) \ll XY^{\varepsilon}Z^{-1}\sum_{\substack{n_1,n_2\leq Y\\n_1n_2=\square}}1,
\end{align}

   Note that upon writing $n_1=r(n_1/r), n_2=r(n_2/r)$ with $r=(n_1, n_2)$, the condition $n_1n_2=\square$ then implies that we must have $n_1/r=s^2_1, n_2/r=s^2_2$. It follows that
\begin{align}
\label{sumsquareterm}
 \sum_{\substack{n_1,n_2\leq Y\\n_1n_2=\square}}1 \ll \sum_{\substack{s_1,s_2\leq \sqrt{Y}}} \sum_{r \leq \min (Y/s^2_1, Y/s^2_2)}1 \ll
\sum_{\substack{s_1,s_2\leq \sqrt{Y}}}\min (\frac Y{s^2_1}, \frac Y{s^2_2}) \ll \sum_{s_1 \leq Y}\Big (\sum_{s_2 \leq s_1}\frac Y{s^2_1}+\sum_{s_2 > s_1}\frac Y{s^2_2} \Big ) \ll Y^{1+\varepsilon}.
\end{align}

   We deduce from \eqref{Oterm} and \eqref{sumsquareterm} that the the $O$-term in \eqref{Sk=01} is
\begin{align}
\label{Otermbound}
\ll XY^{1+\varepsilon}Z^{-1}.
\end{align}

  We next apply the Mellin inversion to recast $H_0(n_1,n_2)$ as
\begin{equation}\label{Mellin1}
    H_0(n_1,n_2)=\Big(\frac{1}{2\pi i}\Big)^2\int_{(2)}\int_{(2)}\frac{Y^uY^v}{n_1^{u}n_2^v}\widetilde{H}_0(u,v)\,du\,dv,
\end{equation}
where
\begin{equation*}
    \widetilde{H}_0(u,v)=\int_{0}^{\infty}\check{h}(xX; uY,vY)d x
\end{equation*}
and $\check{h}(x;u,v)$ is defined by \eqref{checkh}. Similar to \eqref{uppercheckh}, we see that

\begin{equation}\label{estimate1}
    \widetilde{H}_0(u,v)\ll \frac{1}{|uv|(1+|u|)^{A_2} (1 + |v|)^{A_3}}.
\end{equation}

   We deduce from \eqref{Sk=01}, \eqref{Otermbound} and \eqref{Mellin1} that
\begin{align}
\label{Sk=02}
    S_{k=0}
    =\Big(\frac{1}{2\pi i}\Big)^2\frac{4X}{\pi^2}\int_{(2)}\int_{(2)}Y^uY^v\widetilde{H}_0(u,v)Z(u,v)\,du\, dv+O\Big(\frac{XY^{1+\varepsilon}}{Z}\Big),
\end{align}
where
\begin{align*}
    Z(u,v)=\sum_{\substack{(n_1n_2,2)=1\\n_1n_2=\square}}\frac{\lambda_f(n_1)\lambda_f(n_2)}{n_1^{u}n_2^{v}}\prod_{p|n_1n_2}\frac{p}{p+1}.
\end{align*}

  The above function is essentially defined on \cite[p. 1108]{S&Y}, except with $u, v$ replaced by $u+1/2, v+1/2$ there.  We thus deduce from \cite[(4.6)]{S&Y} that
\begin{align*}
%%\label{equal}
        Z(u,v)=\zeta(u+v)L(2u,\operatorname{sym}^2f)L(2v,\operatorname{sym}^2f)L(u+v,\operatorname{sym}^2f)Z_2(u,v),
    \end{align*}
    where $Z_2(u,v)$ converges absolutely in the region of $\Re(u)$ and $\Re(v)$ larger than $1/4$, and is uniformly bounded there.

    We now evaluate the main term of \eqref{Sk=02} by first moving the integrals there to $\Re(u)=\Re(v)=1/2+\varepsilon$ without encountering any pole. We then move the line of integration in $v$ to $\Re(v)=1/4+\varepsilon$ by encountering a simple pole at $v=1-u$, whose residue equals
\begin{align*}
\frac{1}{2\pi i}\frac{4XY}{\pi^2}\int_{(1/2+\varepsilon)}\widetilde{H}_0(u,1-u)L(2u,\operatorname{sym}^2f)L(2-2u,\operatorname{sym}^2f)L(1,\operatorname{sym}^2f)Z_2(u,1-u)\,d u.
\end{align*}

    We now estimate the remaining integrals on $\Re(u)=1/2+\varepsilon$, $\Re(v)=1/4+\varepsilon$ by further moving the line of integration in $u$ to $\Re(u)=1/4+\varepsilon$ without encountering any pole. We then apply the convexity bound for $\zeta(s), L(s,\operatorname{sym}^2f)$ as given in \cite[(5.20)]{iwakow} together with \eqref{SymsquareLfeqn} (which implies that the analytic conductor of $L(s,\operatorname{sym}^2f)$ is $\ll (1+|s|)^3$) to see that for any $\varepsilon>0$,
\begin{align*}
%%\label{zetabound}
   \zeta(s) \ll
   (1+|s|)^{(1-\Re(s))/2+\varepsilon}, \  L(s,\operatorname{sym}^2f) \ll
   (1+|s|)^{3(1-\Re(s))/2+\varepsilon}\qquad & 0< \Re(s) <1
\end{align*}
The above estimations and \eqref{estimate1} with $A_2=A_3=5$ now enable us to see that the integrations on $\Re(u)=1/4+\varepsilon$, $\Re(v)=1/4+\varepsilon$ contribute $\ll XY^{1/2+\varepsilon}$. We thus conclude that
\begin{align}
\label{Sk=03}
\begin{split}
    S_{k=0}
    =&C_0(\Phi, \Psi)XY+O(XY^{1/2+\varepsilon}),
\end{split}
\end{align}
  where
\begin{align}
\label{C0def}
\begin{split}
    C_0(\Phi, \Psi) =: &\frac{1}{2\pi i} \cdot \frac{4}{\pi^2}\int_{(1/2+\varepsilon)}\widetilde{H}_0(u,1-u)L(2u,\operatorname{sym}^2f)L(2-2u,\operatorname{sym}^2f)L(1,\operatorname{sym}^2f)Z_2(u,1-u)\,d u.
\end{split}
\end{align}

\subsection{The $k\neq 0$ terms}

We first state a lemma, which is established in \cite[Sec. 3.3]{S&Y}, to express the weight function appearing in \eqref{S1(h)} in a form more suitable for Mellin transforms.

\begin{lemma}\label{CS}
Let $f$ be a smooth function on $\mathbb{R}_+$ with rapid decay at infinity such that $f$ itself and all its derivatives have finite limits as $x\rightarrow0^+$. We consider the transform given by
\begin{equation*}
    \widehat{f}_{CS}(y) := \int_0^{\infty} f(x) CS(2\pi xy)\, dx,
\end{equation*}
where $CS$ stands for either the $\cos$ or the $\sin$ function.  Then
\begin{equation*}
    \widehat{f}_{CS}(y)  = \frac{1}{2\pi i} \int_{(1/2)} f_M(1-s) \Gamma(s) CS\left(\frac{\sgn(y) \pi s}{2}\right) (2\pi |y|)^{-s} \, ds,
\end{equation*}
  where $f_M(s)$ is the Mellin transform of $f$ defined by
\begin{equation*}
    f_M(s)=\int_{0}^{\infty}x^{s-1}f(x)\,dx.
\end{equation*}
\end{lemma}

   We apply Lemma \ref{CS} to see that
\begin{align}\label{innerint1}
    &\int_{0}^{\infty}(C+S)\Big(\frac{\pi kxX}{a^2n_1n_2}\Big)h(xX,n_1,n_2)\,dx =\frac{1}{(2\pi i)X}\int_{(\varepsilon)}h_M(1-s;n_1,n_2)\Big(\frac{a^2n_1n_2}{\pi|k|}\Big)^s\Gamma(s)(C+\sgn(k)S)\Big(\frac{\pi s}{2}\Big)\,ds.
\end{align}
Taking the Mellin transforms in the variables $n_1$ and $n_2$, the right-hand side of (\ref{innerint1}) equals
\begin{align}
\label{innerint2}
    \frac{1}{X}\Big(\frac{1}{2\pi i}\Big)^3\int_{(\varepsilon)}\int_{(\varepsilon)}\int_{(\varepsilon)} \widetilde{h}_M(1-s,u,v)\frac{1}{n_1^un_2^v}\Big(\frac{a^2n_1n_2}{\pi|k|}\Big)^s\Gamma(s)(C+\sgn(k)S)\Big(\frac{\pi s}{2}\Big)\,ds\,du\,dv,
\end{align}
where
\begin{align*}
    \widetilde{h}_M(s,u,v)=\int_{0}^{\infty}\int_{0}^{\infty}\int_{0}^{\infty}h(x,y,z)x^{s}y^{u}z^{v}\,\frac{dx}{x}\,\frac{dy}{y}\,\frac{dz}{z}.
\end{align*}
Integrating by parts and \eqref{dercond} implies that for $\Re(s)$, $\Re(u)$, $\Re(v)>0$ and any integers $B_j\geq 0$, $1\leq j\leq 3$, we have
\begin{equation}\label{widetildehM}
    \widetilde{h}_M(s,u,v)\ll \frac{X^{\Re(s)}Y^{\Re(u)+\Re(v)}}{|uvs|(1+|s|)^{B_1}(1+|u|)^{B_2}(1+|v|)^{B_3}}.
\end{equation}

   We deduce from \eqref{S1(h)}, \eqref{innerint1} and \eqref{innerint2} that
\begin{align}
\label{Skneq01}
\begin{split}
    S_{k\neq0}&=\frac{1}{2}\sum_{\substack{a\leq Z\\(a,2)=1}}\frac{\mu(a)}{a^2}\sum_{k\neq 0}(-1)^{k}\sum_{\substack{(n_1,2a)=1\\(n_2,2a)=1}}\frac{\lambda_f(n_1)\lambda_f(n_2)}{n_1n_2}G_k(n_1n_2)\Big(\frac{1}{2\pi i}\Big)^3\\
    &\quad\times\int_{(\varepsilon)}\int_{(\varepsilon)}\int_{(\varepsilon)}\widetilde{h}_M(1-s,u,v)\frac{1}{n_1^{u}n_2^{v}}\Big(\frac{a^2n_1n_2}{\pi|k|}\Big)^s\Gamma(s)(C+\sgn(k)S)\Big(\frac{\pi s}{2}\Big)\,ds\,du\,dv.
\end{split}
\end{align}
Note that by \eqref{widetildehM} and the estimation (see \cite[p. 1107]{S&Y}),
\begin{equation}\label{Gammabound}
 \Big| \Gamma(s) (C \pm S)\Big( \frac {\pi s}{2} \Big) \Big| \ll |s|^{\Re(s)-1/2},
\end{equation}
 the integral over $s$ in \eqref{Skneq01} may be taken over any vertical lines between $0$ and $1$ and
the integrals over $u, v$ in \eqref{Skneq01} may be taken over any vertical lines between $0$ and $2$.

Observe that $G_k(m)=G_{4k}(m)$ for odd $m$ and we write $4k = k_1k^2_2$, where $k_1$ is a fundamental discriminant, and $k_2>0$.  Then the expression in  \eqref{Skneq01} can be written as
\begin{align*}
    S_{k\neq 0} =&  \frac{1}{2} \sum_{\substack{a \leq Z \\ (a,2)=1}} \frac{\mu(a)}{a^2} \sumflat_{k_1  } \Big(\frac{1}{2\pi i}\Big)^3\int_{(1+2\varepsilon)} \int_{(1+2\varepsilon)} \int_{(1/2+\varepsilon)}Z(u-s, v-s, s;a,  k_1) \\
&\qquad\qquad \times   \widetilde{h}\left(1-s,u,v \right) \Gamma(s) (C + \sgn(k_1) S)\left(\frac{\pi s}{2} \right)\Big(\frac{a^2}{\pi |k_1|  }\Big)^{s}d s \, d u \, d v,
\end{align*}
 where the function $Z$ is defined in \eqref{eq:Z}.

 We make a change of variables to rewrite $S_{k\neq 0}$ as
\begin{align*}
    S_{k\neq 0}=& \frac{1}{2} \sum_{\substack{a \leq Z \\ (a,2)=1}} \frac{\mu(a)}{a^2} \sumflat_{k_1   }
\left(\frac{1}{2\pi i}\right)^3  \int_{(1/2+\varepsilon)} \int_{(1/2+\varepsilon)}
\int_{(1/2+\varepsilon)} Z(u, v, s;a,  k_1)
\\
\quad &\times   {\widetilde h}(1-s,u+s,v+s)\Gamma(s) (C + \sgn(k_1) S) \left( \frac{\pi s}{2}\right)
 \left(\frac{a^2}{\pi |k_1|}\right)^s
d s \, du \, dv.
\end{align*}
We split the sum over $k_1$ into two terms according to whether $|k_1| \leq K$ or not, with $K$ to be optimized later.  If $|k_1| \leq K$, we move the lines of integration to $\Re(s)= c_1$ for some $1/2<c_1<1$, $\Re(u)=\Re(v)=\varepsilon$. Otherwise, we move
the lines of integration to $\Re(s)=c_2$ for some $c_2>1$, $\Re(u)=\Re(v)=\varepsilon$.
We encounter no poles in either case. Applying Lemma \ref{lemma:Z} and the bound in \eqref{Lbound} yields
\begin{align*}
Z(u, v,s;a,  k_1)
&\ll |L_a( 1/2+u, f \otimes\chi_{ k_1}) L_a(1/2+v,f\otimes \chi_{ k_1})| \ll d^2 (a) \left(|L(1/2+u, f \otimes\chi_{ k_1})|^2 +
|L(1/2+v, f \otimes\chi_{ k_1})|^2\right).
\end{align*}
The above and \eqref{widetildehM} with $B_1=B_2=B_3=5$, together with \eqref{Gammabound} and the symmetry in $u$ and $v$ give that the terms with $|k_1| \leq K$ contribute
\begin{align} \label{k1<K1}
\ll X^{1-c_1} & Y^{2c_1+2\varepsilon}  \sum_{a\leq Z} \frac{d^2(a)}{a^{2-2c_1}}\int_{(\varepsilon)}\int_{(\varepsilon)} \int_{(c_1)}\sumflat_{|k_1|\leq K}\frac{1}{|k_1|^{c_1}}   |L( 1/2+u, f\otimes\chi_{ k_1})|^2 \nonumber\\
&\qquad\qquad\qquad\times  \frac{ |s|^{\Re(s)-1/2}\, ds \,du \, d v }{|1-s||u+s||v+s|(1+|1-s|)^5(1+|u+s|)^5(1+|v+s|)^5}.
\end{align}
  We further apply Lemma \ref{MLL2} and partial summation to get
\begin{align*}
 \sumflat_{ |k_1| \leq K}\frac{1}{|k_1|^{c_1}} |L(1/2+u, f\otimes\chi_{ k_1})|^{2} &\ll K^{1-c_1+\varepsilon}(1+|t|)^{1+\varepsilon} \ll K^{1-c_1+\varepsilon}\left ((1+|u+s|)^{1+\varepsilon}+|s|^{1+\varepsilon} \right ).
\end{align*}
Applying the above in \eqref{k1<K1}, we infer that the terms with $|k_1| \leq K$ contribute
\begin{align*}
%%\label{eq:firstbd1}
\ll X^{1-c_1} Y^{2c_1+2\varepsilon} K^{1-c_1+\varepsilon}Z^{2c_1-1+\varepsilon}.
\end{align*}

Similarly, the contribution from the complementary terms with $|k_1| >K$  is
\begin{align*}
\ll X^{1-c_2} Y^{2c_2+2\varepsilon} K^{1-c_2+\varepsilon}Z^{2c_2-1+\varepsilon}.
\end{align*}
  We now balance these contributions by setting $K=Y^2Z^2/X$ so that
\begin{align*}
 X^{1-c_1} Y^{2c_1} K^{1-c_1}Z^{2c_1-1}= X^{1-c_2} Y^{2c_2} K^{1-c_2}Z^{2c_2-1}.
\end{align*}
Now taking $c_1=1/2+\varepsilon$, we conclude
\begin{align}\label{Skneq02}
S_{k\neq 0} \ll (XYZ)^{\varepsilon}Y^2Z.
\end{align}
  We now deduce from \eqref{S(h)},  Lemma \ref{S2h}, \eqref{Sk=03} and \eqref{Skneq02} that
\begin{align*}
 S(h) = C_0(\Phi, \Psi)XY + O\left( (XY)^{1 + \varepsilon} Z^{-1+ \varepsilon}+(XYZ)^{\varepsilon}Y^2Z+XY^{1/2+\varepsilon}   \right).
\end{align*}
  Now the expression in \eqref{S} follows upon setting $Z=\sqrt{X/Y}$ in the above expression. This completes the proof of Theorem \ref{meansquare}.

\vspace*{.5cm}

\noindent{\bf Acknowledgments.}  P. G. is supported in part by NSFC grant 12471003.

\bibliography{biblio}
\bibliographystyle{amsxport}

\end{document}